\documentclass[11pt]{article}
\usepackage[T1]{fontenc}
\usepackage[utf8]{inputenc}
\usepackage[english]{babel}
\usepackage{microtype}
\usepackage{amsmath,amssymb,amsfonts,amsthm,mathtools}
\usepackage[shortlabels]{enumitem}
\usepackage{geometry}
\usepackage[colorlinks=true,linkcolor=blue,citecolor=blue,urlcolor=blue]{hyperref}
\newtheorem{theorem}{Theorem}[section]
\newtheorem{lemma}[theorem]{Lemma}
\newtheorem{claim}[theorem]{Claim}

\newtheorem{corollary}[theorem]{Corollary}
\newtheorem{definition}[theorem]{Definition}

\newtheorem{problem}[theorem]{Problem}

\newcommand{\eps}{\varepsilon}

\newcommand{\cE}{\mathcal E}

\newcommand{\floor}[1]{\lfloor #1\rfloor}
\newcommand{\ceil}[1]{\lceil #1\rceil}

\newcommand{\1}{\mathbf 1}
\newcommand{\dd}{\,d}
\newcommand{\diag}{\operatorname{diag}}

\usepackage{xcolor}
\newcommand{\caiedit}[1]{\textcolor{red}{#1}}
\usepackage{extarrows}
\title{Sharp Bounds for Guiduli-Type Hereditary Spectral Problems}

\author{
Dongxiu Cai\thanks{Email: \href{mailto:diudiutse@sjtu.edu.cn}{diudiutse@sjtu.edu.cn}},\quad
Jiasheng Zeng\thanks{Email: \href{mailto:jasonzeng@mail.ustc.edu.cn}{jasonzeng@mail.ustc.edu.cn}},\quad
and Xiao-Dong Zhang\thanks{Email: \href{mailto:xiaodong@sjtu.edu.cn}{xiaodong@sjtu.edu.cn}}
}
\date{\today}

\begin{document}

\maketitle

\begin{abstract}
Guiduli asked in 1996 the following problem concerning the maximum spectral radius of a graph under hereditary density constraints. If an $n$-vertex graph $G$ satisfies $e(H)\le c|V(H)|^2$ for every subgraph $H$ of $G$, must one have $\lambda(G)\le 2cn$? More generally, what remains true when the exponent $2$ is replaced by a constant less than $2$? We study the natural power-law version of this question for all $1<p\le2$. For $1<p\le 2$, define
\[
        d_p(G)=\max_{\varnothing\ne S\subseteq V(G)}\frac{e(G[S])}{|S|^p}.
\]
We determine the sharp asymptotic upper bound for $\lambda(G)$ in terms of $d_p(G)$ and $n$. More precisely, every $n$-vertex graph $G$ with at least one edge satisfies
\[
\lambda(G)\le
\begin{cases}
\left(\left(\max_{t\in\mathbb N_{\ge1}}\dfrac{t}{(t+1)^p}\right)^{-1}+o(1)\right)d_p(G)\sqrt n,&1<p<3/2,\\[0.4em]
\left(\dfrac{3\sqrt3}{4}+o(1)\right)d_p(G)\sqrt{n\log n},&p=3/2,\\[0.4em]
(\mathfrak C_p+o(1))d_p(G)n^{p-1},&3/2<p<2,
\end{cases}
\]
and each constant here is best possible. Here $\mathfrak C_p$ is characterized by an exact variational problem over finite kernels. We apply a sparse graphon operator estimate to convert hereditary $p$-density bounds into sharp spectral bounds, and this estimate also explains the transition at the critical exponent $p=3/2$. For the endpoint $p=2$, Wilf's theorem gives the exact finite-$n$
bound $\lambda(G)\le 2d_2(G)n$, with equality for $K_n$. Thus Guiduli's power-law problem is resolved in its sharp asymptotic form for every $1<p\leq2$, including exact leading constants. 
\end{abstract}

\section{Introduction}
Let $G$ be a simple graph. The spectral radius of $G$, denoted as $\lambda(G)$, is the largest eigenvalue of its adjacency matrix $A(G)$. A central topic in spectral extremal graph theory is to determine the maximum  spectral radius of a graph under prescribed structural restrictions and to characterize the corresponding extremal graphs. The spectral Tur\'an problem asks, for a given family of forbidden graphs and a fixed number of vertices, how large the spectral radius can be among all graphs avoiding that family, and may be regarded as a spectral counterpart of the classical theorems of Mantel \cite{Mantel1907} and Tur\'an \cite{Turan1941}. Early results of Nosal \cite{Nosal1970}, Wilf \cite{Wilf1986} and Stanley \cite{Stanley1987} showed that a sufficiently large spectral radius forces the presence of certain dense substructures, thereby initiating the use of eigenvalue methods in extremal graph theory. Nikiforov developed this direction further by proving spectral Tur\'an-type theorems for $K_{r+1}$-free graphs \cite{Nikiforov2007} and by studying the maximum spectral radius of graphs with no paths or cycles of prescribed lengths \cite{Nikiforov2010PathsCycles}. Babai and Guiduli proved a spectral Zarankiewicz theorem, giving spectral analogues of K\H{o}vari--S\'{o}s--Tur\'{a}n type bounds for complete bipartite forbidden graphs \cite{BabaiGuiduli2009}. More recently, Cioab\u{a}, Desai and Tait resolved the spectral even cycle conjecture of Nikiforov \cite{CioabaDesaiTait2024}. Broader structural results and conjectures in this area also appear in the work of Tait and Tobin \cite{TaitTobin2017}, Wang, Kang and Xue \cite{WangKangXue2023}, and Byrne, Desai and Tait \cite{ByrneDesaiTait2024}. We also refer the reader to the excellent survey of Liu and Ning \cite{LiuNing2023} for recent progress, open problems and conjectures in spectral extremal graph theory.

Guiduli \cite{Guiduli1996} proposed the following three problems in 1996, which differ from the classical forbidden-subgraph setting in that they investigate how local density and hereditary density conditions constrain the spectral radius of a graph.

\begin{problem}[Guiduli-type hereditary density spectral problems \cite{Guiduli1996}]\label{prob:guiduli-family}
Let $G$ be a finite graph.
\begin{enumerate}[(1)]
    \item Let $t$ and $r$ be fixed. Suppose that every subgraph $H\subseteq G$ with $|V(H)|\ge t$ satisfies
    $
        e(H)\le t|V(H)|+r.
    $
    Determine the sharp upper bound for $\lambda(G)$ and characterize the extremal graphs.

    \item Let $r\ge2$ and let $T_r(n)$ be the $r$-partite Tur\'an graph. Suppose that $G$ is an $n$-vertex graph with
    $
        \lambda(G)\ge \lambda(T_r(n)).
    $
    Must it be true that either $G\cong T_r(n)$ or there exists a vertex $v\in V(G)$ such that
    $
        \lambda(G[N(v)])>\lambda(T_{r-1}(d(v)))?
    $

    \item Let $G$ be an $n$-vertex graph and let $c>0$. Suppose that
    $
        e(H)\le c|V(H)|^2
    $
    for every subgraph $H$ of $G$. Must one have
    $
        \lambda(G)\le 2cn?
    $
    More generally, what remains true when the exponent $2$ is replaced by a constant less than $2$?
\end{enumerate}
\end{problem}

Problem~\ref{prob:guiduli-family}(1) was recently resolved by Li, Wang and Zhai, who also characterized the corresponding extremal graphs \cite{LiWangZhai2024}. Problem~\ref{prob:guiduli-family}(2) was recently proved in a stronger form by Liu and Ning, who in the same work also solved two problems of Nikiforov \cite{LiuNing2026}. In the present paper, we study Problem~\ref{prob:guiduli-family}(3) and give a sharp solution to its natural power-law extension.

For $1<p\le 2$, define $$d_p(G)=\max_{\varnothing\ne S\subseteq V(G)} e(G[S])/|S|^p.$$ Thus the condition $d_p(G)\le D$ means that every induced subgraph of $G$ satisfies the power-law density bound $e(G[S])\le D|S|^p$. Since every subgraph with vertex set $S$ has at most $e(G[S])$ edges, imposing this condition on induced subgraphs is equivalent, for the upper-bound problems considered in this paper, to imposing it on all subgraphs. We determine the sharp asymptotic value of the following extremal quantity
$$\Lambda_p(n)=\sup_{|V(G)|=n,\ e(G)>0} \lambda(G)/d_p(G).$$

We first state our main theorem for finite graphs. In the range $3/2<p<2$, the leading constant is the value of a variational problem over finite kernels.

\begin{definition}\label{def:Cp}
For $3/2<p<2$, define
\[
\mathfrak C_p=\sup_{r,a,K}\frac{2\lambda_{\max}(D_a^{1/2}KD_a^{1/2})}{\max_{\substack{0\le s_i\le a_i\\ \sum_i s_i>0}}\dfrac{s^{\mathsf T}Ks}{(\sum_i s_i)^p}},
\]
where $r\ge1$, $a=(a_1,\ldots,a_r)$ ranges over vectors with $a_i>0$ and $\sum_i a_i=1$, $D_a=\diag(a_1,\ldots,a_r)$, and $K=K^{\mathsf T}$ ranges over nonzero symmetric nonnegative $r\times r$ matrices. Since the quotient is homogeneous in $K$, one may equivalently impose $0\le K_{ij}\le1$.
\end{definition}


\begin{theorem}\label{thm:main}
For $1<p\le2$, define
\[
\Lambda_p(n)=\sup_{|V(G)|=n,\ e(G)>0}\frac{\lambda(G)}{d_p(G)},\qquad d_p(G)=\max_{\varnothing\ne S\subseteq V(G)}\frac{e(G[S])}{|S|^p}.
\]
Then
\[
\Lambda_p(n)=\begin{cases}\left(A_p+o(1)\right)\sqrt n,&1<p<3/2,\\ \left(\dfrac{3\sqrt3}{4}+o(1)\right)\sqrt{n\log n},&p=3/2,\\ \left(\mathfrak C_p+o(1)\right)n^{p-1},&3/2<p<2,\\ 2n,&p=2,\end{cases}
\]
where
\[
A_p=\left(\max_{t\in\mathbb N_{\ge1}}\frac{t}{(t+1)^p}\right)^{-1}.
\]
The maximum defining $A_p$ is attained by one of $t=\floor{1/(p-1)}$ and $t=\ceil{1/(p-1)}$.
\end{theorem}

The proof of Theorem~\ref{thm:main} uses a sparse graphon operator estimate. This estimate provides the analytic mechanism behind the three scales in the theorem, while the sharp constants require additional finite-graph arguments and matching constructions in the three regimes. Before we state the theorem, we first give some basic notations.

Throughout the paper, $C_p$ denotes a positive constant depending only on $p$,
whose value may change from line to line. A kernel means a symmetric nonnegative measurable function on an atomless probability space. Given an $n$-vertex graph $G$ and $1<p<2$, we use the rescaled kernel $W_{G,p}=n^{2-p}W_G$, where $W_G$ is the usual adjacency step graphon. Then $\|T_{W_{G,p}}\|_{2\to2}=\lambda(G)/n^{p-1}$. The hereditary bound $e(G[S])\le D|S|^p$ becomes a graphon inequality of the form $\Delta_p(W_{G,p})\le C_pD$ (we will define $\Delta_p$ later), while the finite resolution of the graph is recorded by the height $\|W_{G,p}\|_\infty=n^{2-p}$. The height is essential for order estimates. The sharp constants below require a finer finite-kernel normalization in the supercritical range and a separate Perron-level analysis in the subcritical and critical ranges. 

\begin{definition}\label{def:kernel-density}
Let $(\Omega,\mu)$ be a probability space, let $W:\Omega^2\to[0,\infty)$ be a symmetric measurable kernel, and let $1<p\le2$. Define
\[
\Delta_p(W)=\sup_{\mu(A)>0}\frac{\int_{A\times A}W(x,y)\dd\mu(x)\dd\mu(y)}{\mu(A)^p}.
\]
For $1<p<2$ and $R\ge1$, define
\[
\Phi_p(R)=\begin{cases}R^{\frac{3/2-p}{2-p}},&1<p<3/2,\\ \sqrt{\log(eR)},&p=3/2,\\ 1,&3/2<p<2.\end{cases}
\]
\end{definition}

\begin{theorem}\label{thm:graphon}
For every $1<p<2$ there is a constant $C_p$ such that the following holds. If $W$ is a symmetric nonnegative kernel on an atomless probability space, $0<\Delta_p(W)<\infty$, and $\|W\|_\infty<\infty$, then
\[
\|T_W\|_{2\to2}\le C_p\Delta_p(W)\Phi_p\left(1+\frac{\|W\|_\infty}{\Delta_p(W)}\right).
\]
If $\Delta_p(W)=0$, then $W=0$ almost everywhere and the same conclusion holds with right-hand side $0$.
\end{theorem}

The following corollary is the form closest to Guiduli's original question.

\begin{corollary}\label{cor:finite-orders}
Let $1<p\le2$. If $G$ is an $n$-vertex graph and $e(G[S])\le D|S|^p$ for every $S\subseteq V(G)$, then
\[
\lambda(G)\le\begin{cases}(A_p+o(1))D\sqrt n,&1<p<3/2,\\ \left(\dfrac{3\sqrt3}{4}+o(1)\right)D\sqrt{n\log n},&p=3/2,\\ (\mathfrak C_p+o(1))Dn^{p-1},&3/2<p<2,\\ 2Dn,&p=2,\end{cases}
\]
where the $o(1)$ terms are uniform over all $n$-vertex graphs with at least one edge after $p$ is fixed. Each constant displayed in Theorem~\ref{thm:main} is best possible in the corresponding regime.
\end{corollary}

The endpoint $p=2$ is controlled by Wilf's spectral clique bound~\cite{Wilf1986}. Let $\omega=\omega(G)$. Since a clique $K_\omega$ gives $d_2(G)\ge \binom{\omega}{2}/\omega^2=(1-1/\omega)/2$, while Wilf's theorem gives $\lambda(G)\le (1-1/\omega)n$, it follows that $\lambda(G)\le 2d_2(G)n$. Equality is attained by $K_n$, and hence the quadratic case of Problem~\ref{prob:guiduli-family}(3) has the exact finite answer $\Lambda_2(n)=2n$.

Graphon methods provide the analytic language used in this paper. Dense graphons were introduced by Lov\'asz and Szegedy \cite{LovaszSzegedy2006}; for further background, we refer the reader to the monograph of Lov\'asz \cite{Lovasz2012}. Graphons have become a fundamental tool in extremal graph theory, graph limits and random graphs. In the spectral direction, Szegedy studied limits of kernel operators and proved a spectral regularity lemma \cite{Szegedy2011}. Bollob\'as, Borgs, Chayes and Riordan used graph limits together with the largest eigenvalue to determine percolation thresholds for dense graph sequences \cite{BollobasBorgsChayesRiordan2010}, while Bollob\'as, Janson and Riordan developed a kernel-operator approach to phase transitions in inhomogeneous random graphs \cite{BollobasJansonRiordan2007}. 

The use of graph limits in spectral extremal problems has also developed rapidly in recent years. Liu \cite{Liu2024GraphLimitsSpectral} used graph limits to resolve spectral extremal conjectures concerning linear combinations of graph eigenvalues, including the Nordhaus--Gaddum type extremal problem for $\lambda_1(G)+\lambda_1(\overline G)$ and the extremal problem for the signless Laplacian $Q$-spread. Breen, Riasanovsky, Tait and Urschel used graph limits, operator methods and numerical analysis to study the adjacency spread, namely the difference between the largest and smallest eigenvalues of the adjacency matrix, and resolved asymptotic versions of longstanding conjectures on the maximum spread of graphs and bipartite graphs \cite{BreenRiasanovskyTaitUrschel2022}. Kumar, Liu, Monterde, Pragada and Tait \cite{KumarLiuMonterdePragadaTait2026} recently studied the maximum of the spectral sum $\lambda_1(G)+\lambda_2(G)$, and combined graph limits, convex geometry, exterior algebra and convex optimization to prove a conjecture of Ebrahimi, Mohar, Nikiforov and Ahmady. In the direction of large-scale networks and sampled graphs, Vizuete, Garin and Frasca \cite{VizueteGarinFrasca2021} studied the Laplacian spectrum of large graphs sampled from graphons, showing that, under suitable assumptions, the Laplacian eigenvalues of sampled graphs and related network quantities can be approximated by spectral data of the underlying graphon. Gao and Caines \cite{GaoCaines2019} studied spectral representations, eigenfunctions and spectral approximations of graphons from the perspective of control for large-scale network systems, illustrating the role of the spectral structure of graphons as Hilbert--Schmidt integral operators in the analysis of very large network systems.
Sparse graphon theories, especially the $L^p$ theory of Borgs, Chayes, Cohn and Zhao \cite{BorgsChayesCohnZhao2018,BorgsChayesCohnZhao2019}, provide a framework for sparse limits with unbounded average degree. Graphon and density methods have also led to important advances in extremal graph theory, including flag algebras \cite{Razborov2007}, triangle and clique density theorems \cite{Razborov2008,Reiher2016}, and undecidability phenomena for homomorphism density inequalities \cite{HatamiNorine2011}.

\section{Notation and preliminaries}

All graphs are finite and simple. For $S\subseteq V(G)$, let $G[S]$ be the induced subgraph and let $e(S)=e(G[S])$. The quantity $d_p(G)$ is unchanged if one requires $e(H)\le D|V(H)|^p$ for every subgraph $H$ instead of every induced subgraph, since every subgraph on $S$ has at most $e(G[S])$ edges. For an $n$-vertex graph $G$, the associated step graphon $W_G$ is obtained by partitioning $[0,1]$ into intervals $I_v$ of length $1/n$ and setting $W_G(x,y)=1$ when $x\in I_u$, $y\in I_v$ and $uv\in E(G)$, and $W_G(x,y)=0$ otherwise.

For vertex sets $A,B\subseteq V(G)$, write
\[
\cE(A,B)=\sum_{u\in A,\ v\in B}A_{uv}.
\]
Thus $\cE(A,A)=2e(A)$, and if $A\subseteq B$, then $\cE(A,B)=2e(A)+e(A,B\setminus A)$. This convention is used in all graph layer-cake identities below.

If $W$ is a bounded symmetric nonnegative kernel on a probability space $(\Omega,\mu)$, then $T_W$ denotes the integral operator $T_Wf(x)=\int_\Omega W(x,y)f(y)\dd\mu(y)$ on $L^2(\Omega)$. Since $W\ge0$, one has $|\langle T_Wf,f\rangle|\le\langle T_W|f|,|f|\rangle$, and hence the operator norm is obtained by considering nonnegative test functions in Rayleigh quotients.

We use repeatedly the following layer-cake notation. For a nonnegative measurable function $f$ on $\Omega$, set $S_t(f)=\{x:f(x)\ge t\}$ and $N_f(t)=\mu(S_t(f))$. If $\|f\|_2=1$, then
\[
\int_0^\infty 2tN_f(t)\dd t=1,
\]
and $\int_0^\infty N_f(t)\dd t=\|f\|_1\le1$.

\begin{lemma}\label{lem:layercake}
Let $W$ be a nonnegative kernel and let $f,g\ge0$ be measurable functions. Then
\[
\int_{\Omega^2}W(x,y)f(x)g(y)\dd\mu(x)\dd\mu(y)=\int_0^\infty\int_0^\infty \int_{S_s(f)\times S_t(g)}W(x,y)\dd\mu(x)\dd\mu(y)\dd s\dd t,
\]
where both sides are allowed to be $+\infty$.
\end{lemma}

\begin{proof}
This is Tonelli's theorem applied to $f(x)g(y)=\int_0^\infty\int_0^\infty \1_{f(x)\ge s}\1_{g(y)\ge t}\dd s\dd t$.
\end{proof}

\begin{lemma}\label{lem:nested-incidence}
Let $1<p<2$, let $W$ be a symmetric nonnegative kernel on an atomless probability space with $\Delta_p(W)\le1$ and $\|W\|_\infty\le R$, where $R\ge1$. If $A\subseteq B\subseteq\Omega$, $a=\mu(A)$ and $b=\mu(B)$, then
\[
\int_{A\times B}W\le C_p b\min\{Ra,a^{p-1}\}.
\]
\end{lemma}

\begin{proof}
If $a=0$, then there is nothing to prove. The bound $\int_{A\times B}W\le Rab$ follows immediately from $\|W\|_\infty\le R$. It remains to prove $\int_{A\times B}W\le C_pb a^{p-1}$. Define $h(y)=\int_A W(x,y)\dd\mu(x)$ on $B$, and let $h^*$ be the decreasing rearrangement of $h$ on $[0,b]$. Since the underlying probability space is atomless, the Hardy--Littlewood maximal property of decreasing rearrangements gives
\[
\int_0^t h^*(u)\dd u=\sup_{\substack{E\subseteq B\\ \mu(E)=t}}\int_Eh\dd\mu
\]
for every $0<t\le b$. For every measurable $E\subseteq B$ with $\mu(E)=t$, nonnegativity and $\Delta_p(W)\le1$ give $\int_Eh\dd\mu=\int_{A\times E}W\le\int_{(A\cup E)\times(A\cup E)}W\le(a+t)^p$. Hence $\int_0^t h^*(u)\dd u\le(a+t)^p$ for every $0<t\le b$. Taking $t=a$ gives $h^*(a)\le a^{-1}\int_0^a h^*(u)\dd u\le2^pa^{p-1}$. Since $A\subseteq B$, we have $a\le b$, and therefore
\[
\int_Bh\dd\mu=\int_0^b h^*(u)\dd u\le(2a)^p+(b-a)2^pa^{p-1}\le C_pb a^{p-1}.
\]
This proves the lemma.
\end{proof}

\begin{lemma}\label{lem:hardy}
Let $1<p<2$, let $R\ge1$, and put $\phi_R(u)=\min\{Ru,u^{p-1}\}$ for $0\le u\le1$. Suppose $N:[0,\infty)\to[0,1]$ is measurable and $\int_0^\infty2tN(t)\dd t\le1$. Let $F(t)=\int_0^tN(s)\dd s$. Then
\[
\int_0^\infty F(t)\phi_R(N(t))\dd t\le C_p\Phi_p(R).
\]
\end{lemma}

\begin{proof}
The contribution of $0\le t\le1$ is at most $\int_0^1t\dd t\le1/2$, since $F(t)\le t$ and $\phi_R(N(t))\le1$. For $j\ge0$, let $I_j=[2^j,2^{j+1}]$, $T_j=2^j$, and $E_j=\int_{I_j}2tN(t)\dd t$. Then $\sum_jE_j\le1$. The function $\phi_R$ is increasing and concave. It is linear on $[0,a_0]$ and equal to $u^{p-1}$ on $[a_0,1]$, where $a_0=R^{-1/(2-p)}$, and the right derivative at $a_0$ is smaller than the left derivative. Since $t\asymp T_j$ on $I_j$, Jensen's inequality gives
\[
\int_{I_j}\phi_R(N(t))\dd t\le C T_j\phi_R\left(C\frac{E_j}{T_j^2}\right).
\]
Absorbing the absolute constant into $C_p$, and using $F(t)\le\int_0^\infty N(u)\dd u\le3/2$, it remains to bound $\sum_jT_j\phi_R(E_j/T_j^2)$. The last inequality follows from $\int_0^1N(u)\dd u\le1$ and $\int_1^\infty N(u)\dd u\le\int_1^\infty uN(u)\dd u\le1/2$. Split the indices into those for which $E_j/T_j^2\le a_0$ and those for which $E_j/T_j^2>a_0$. In the first range the contribution is $R\sum E_j/T_j$. Under the constraints $0\le E_j\le a_0T_j^2$ and $\sum E_j\le1$, the sum $\sum E_j/T_j$ is maximized by filling the smallest dyadic scales first. If $J$ is chosen by $T_J\le a_0^{-1/2}<T_{J+1}$, then
\[
\sum_{E_j/T_j^2\le a_0}\frac{E_j}{T_j}\le\sum_{j\le J}a_0T_j+\frac{1}{T_J}\sum_{j>J}E_j\le C a_0^{1/2}.
\]
Thus the first contribution is at most $CRa_0^{1/2}=CR^{(3/2-p)/(2-p)}$, which is $O(1)$ when $p\ge3/2$. In the second range, writing $q=p-1$, one obtains
\[
\sum_{E_j/T_j^2>a_0}T_j\left(\frac{E_j}{T_j^2}\right)^q=\sum_{E_j/T_j^2>a_0}E_j^{p-1}T_j^{3-2p}.
\]
The condition $E_j/T_j^2>a_0$ and the inequality $E_j\le1$ imply $T_j\le a_0^{-1/2}$. Holder's inequality with exponents $1/(p-1)$ and $1/(2-p)$ yields
\[
\sum_{E_j/T_j^2>a_0}E_j^{p-1}T_j^{3-2p}\le\left(\sum_jE_j\right)^{p-1}\left(\sum_{T_j\le a_0^{-1/2}}T_j^{\frac{3-2p}{2-p}}\right)^{2-p}.
\]
If $p>3/2$, the dyadic sum is bounded. If $p=3/2$, it has $O(\log(eR))$ terms and the final exponent is $2-p=1/2$, giving $O(\sqrt{\log(eR)})$. If $1<p<3/2$, the dyadic sum is dominated by its largest term and gives $O(R^{(3/2-p)/(2-p)})$. Combining the two estimates proves the lemma.
\end{proof}

\begin{lemma}\label{lem:finite-graphon}
Let $1<p<2$ and let $G$ be an $n$-vertex graph with at least one edge. Set $W_{G,p}=n^{2-p}W_G$. Then $\|W_{G,p}\|_\infty=n^{2-p}$, $\|T_{W_{G,p}}\|_{2\to2}=\lambda(G)/n^{p-1}$, and
\[
2d_p(G)\le \Delta_p(W_{G,p})\le C_pd_p(G).
\]
\end{lemma}

\begin{proof}
The height identity is immediate. The operator $T_{W_G}$ maps every function to a function constant on the vertex intervals, and on that finite-dimensional subspace it is represented by $A(G)/n$. Hence $\|T_{W_{G,p}}\|_{2\to2}=n^{2-p}\lambda(G)/n=\lambda(G)/n^{p-1}$. For the lower bound on $\Delta_p(W_{G,p})$, let $U\subseteq V(G)$ attain $d_p(G)$ and put $A_U=\bigcup_{v\in U}I_v$. Then
\[
\frac{\int_{A_U\times A_U}W_{G,p}}{\mu(A_U)^p}=\frac{2n^{2-p}e(G[U])/n^2}{(|U|/n)^p}=\frac{2e(G[U])}{|U|^p}=2d_p(G).
\]
It remains to prove the upper bound. Let $A\subseteq[0,1]$ be measurable and write $\theta_v=n\mu(A\cap I_v)\in[0,1]$ and $S=\sum_v\theta_v=n\mu(A)$. Then
\[
\int_{A\times A}W_{G,p}=2n^{-p}\sum_{uv\in E(G)}\theta_u\theta_v.
\]
If $S<1$, then $\sum_{uv\in E(G)}\theta_u\theta_v\le S^2/2\le S^p/2$, and since $G$ has an edge, $d_p(G)\ge2^{-p}$, so the desired estimate follows. If $S\ge1$, let $R$ be the random vertex set obtained by selecting each vertex $v$ independently with probability $\theta_v$. Then $\mathbb E e(G[R])=\sum_{uv\in E(G)}\theta_u\theta_v$, while $e(G[R])\le d_p(G)|R|^p$ for every outcome. Since $1<p<2$ and $\mathbb E|R|=S\ge1$, we have $$\mathbb E|R|^p\le(\mathbb E|R|^2)^{p/2}\le(S^2+S)^{p/2}\le C_pS^p.$$ Therefore $\sum_{uv\in E(G)}\theta_u\theta_v\le C_pd_p(G)S^p$, and substituting into the displayed identity gives $\int_{A\times A}W_{G,p}\le C_pd_p(G)\mu(A)^p.$
\end{proof}

\begin{lemma}\label{lem:block-averaging}
Let $x_1\ge x_2\ge\cdots\ge x_n\ge0$, and let $P_1,\ldots,P_r$ be consecutive intervals in $\{1,\ldots,n\}$ of sizes at most $2M$. If $y$ is obtained from $x$ by replacing $x$ on each $P_i$ by its average on $P_i$, then
\[
\sum_i\sum_{v\in P_i}(x_v-y_v)^2\le 2M\|x\|_\infty^2.
\]
\end{lemma}

\begin{proof}
For a block $P_i=[a_i,b_i]$, every entry of the block lies between $x_{a_i}$ and $x_{b_i}$, so $\sum_{v\in P_i}(x_v-y_v)^2\le |P_i|(x_{a_i}-x_{b_i})^2\le2M(x_{a_i}-x_{b_i})^2$. The block intervals are disjoint and ordered, hence $\sum_i(x_{a_i}-x_{b_i})\le x_1\le\|x\|_\infty$, and the displayed estimate follows.
\end{proof}

\begin{lemma}\label{lem:tail}
Let $1<p<3/2$. Let $G$ be an $n$-vertex graph, let $D=d_p(G)$, and let $w\in\mathbb R_{\ge0}^{V(G)}$ satisfy $\|w\|_2\le1$ and $\|w\|_\infty\le\eps$. Then
\[
w^{\mathsf T}A(G)w\le C_pD\left(\eps^{3-2p}\sqrt n+n^{p-1}\right).
\]
\end{lemma}

\begin{proof}
For $t>0$, let $S_t=\{v:w_v\ge t\}$ and $N(t)=|S_t|$. Notice that when $s\le t$, we have   $S_t\subseteq S_s$ and 
$$\begin{aligned}&\quad \int_{0}^\epsilon \int_{0}^\epsilon  \cE(S_t,S_s)\dd s\dd t
\\&=\int_{0}^\epsilon \int_{0}^\epsilon  \sum_{u,v\in V(G)}1_{u\sim v}\1_{w_u\geq t}\1_{w_v\geq s}\dd s\dd t
\\&=\sum_{u,v\in V(G)}\1_{u\sim v}w_uw_v=w^{\mathsf T}Aw.\end{aligned}$$
Moreover, since $\cE(S_t,S_s)$ is symmetric in $s$ and $t$, we have $$\int_{0}^\epsilon \int_{0}^\epsilon  \cE(S_t,S_s)\dd s\dd t=2\int_{0}^\epsilon \int_{0}^t  \cE(S_t,S_s)\dd s\dd t.$$
Therefore, we obtain that $$w^{\mathsf T}Aw=2\int_0^\eps\int_0^t \cE(S_t,S_s)\dd s\dd t.$$
Now we control the upper bound of $\cE(S_t,S_s)$. Since
$$
    \cE(S_t,S_s)=\sum_{v\in S_s}\cE(S_t,v)=\sum_{i=1}^{N(s)}\cE(S_t,v_i),
$$where $S_s=\{v_1,\ldots,v_{N(s)}\}$ such that $\cE(S_t,v_1)\geq\cdots\geq \cE(S_t,v_{N(s)})$. Moreover,
$$\begin{aligned}\sum_{i=1}^{N(t)}\cE(S_t,v_i)&=\cE(S_t,\{v_1\ldots,v_{N(t)}\})
\\&\leq \cE(S_t\cup\{v_1\ldots,v_{N(t)}\},S_t\cup\{v_1\ldots,v_{N(t)}\})
\\&\leq 2D(N(t)+N(t))^p=2^{p+1}DN(t)^p.\end{aligned}$$
Hence $\cE(S_t,v_{N(t)})\leq N(t)^{-1}\sum_{i=1}^{N(t)}\cE(S_t,v_i)\leq 2^{p+1}DN(t)^{p-1}$. Therefore, $$\begin{aligned}
\cE(S_t,S_s)&=\sum_{i=1}^{N(s)}\cE(S_t,v_i)
\leq \sum_{i=1}^{N(t)}\cE(S_t,v_i)+\cE(S_t,v_{N(t)})(N(s)-N(t))
\\&\leq 2^{p+1}DN(t)^p+2^{p+1}DN(t)^{p-1}(N(s)-N(t))
\\&=2^{p+1}DN(t)^{p-1}N(s)
\le C_pD N(s)N(t)^{p-1}.
\end{aligned}$$
Hence
\[
w^{\mathsf T}Aw\le C_pD\int_0^\eps F(t)N(t)^{p-1}\dd t,\qquad F(t)=\int_0^tN(s)\dd s.
\]
Since $\|w\|_2\le1$, one has $N(t)\le\min(n,t^{-2})$. For $0<t<n^{-1/2}$, $F(t)\le nt$, and the corresponding contribution is at most $C_pDn^{p-1}$. For $n^{-1/2}<t<\eps$, $F(t)\le C\sqrt n$ and $N(t)\le t^{-2}$, so the contribution is at most
\[
C_pD\sqrt n\int_{n^{-1/2}}^\eps t^{-2p+2}\dd t\le C_pD\eps^{3-2p}\sqrt n.
\]
This proves the estimate.
\end{proof}

\begin{lemma}\label{lem:critical-hardy}
Let $N:[0,\infty)\to[0,n]$ be measurable, assume $N(t)=0$ for $t>1$, $N(t)\le t^{-2}$ for $t>0$ and $\int_0^\infty2tN(t)\dd t\le1$, and set $F(t)=\int_0^tN(s)\dd s$. Then
\[
\int_0^\infty F(t)\sqrt{N(t)}\dd t\le\left(\frac14+o(1)\right)\sqrt{n\log n},
\]
where $o(1)$ tends to $0$ as $n\to\infty$ uniformly over all such $N$.
\end{lemma}

\begin{proof}
Let $L=\log n$ and define $t(u)=e^u/\sqrt n$, $Q(u)=t(u)^2N(t(u))$ and $f(u)=F(t(u))/\sqrt n$.

\begin{claim}\label{claim:critical-change}
We have $Q(u)=0$ for $u>L/2$,
\[
        0\le Q(u)\le1,\qquad Q(u)\le e^{2u},\qquad \int_{-\infty}^{L/2}2Q(u)\dd u\le1,
\]
and
\[
        f(u)=\int_{-\infty}^uQ(v)e^{-v}\dd v,
        \qquad
        \frac1{\sqrt n}\int_0^\infty F(t)\sqrt{N(t)}\dd t
        =
        \int_{-\infty}^{L/2}f(u)\sqrt{Q(u)}\dd u.
\]
\end{claim}

\begin{proof}[Proof of Claim~\ref{claim:critical-change}]
Since $u>L/2$ implies $t(u)>1$ and $N(t)=0$ for $t>1$, we have $Q(u)=0$ for $u>L/2$. Since $N(t)\le t^{-2}$, one has $Q(u)=t(u)^2N(t(u))\le1$. Since $N(t)\le n$, one also has $Q(u)=t(u)^2N(t(u))\le nt(u)^2=e^{2u}$. The change of variables $t=e^u/\sqrt n$ gives
\[
        \int_{-\infty}^{L/2}2Q(u)\dd u
        =
        \int_0^1 2tN(t)\dd t
        \le1.
\]
The same change of variables gives
\[
        f(u)=\frac1{\sqrt n}\int_0^{t(u)}N(s)\dd s
        =
        \int_{-\infty}^uQ(v)e^{-v}\dd v.
\]
Finally, using $N(t)=0$ for $t>1$ and again changing variables by $t=e^u/\sqrt n$, we obtain
\[
        \frac1{\sqrt n}\int_0^\infty F(t)\sqrt{N(t)}\dd t
        =
        \frac1{\sqrt n}\int_0^1F(t)\sqrt{N(t)}\dd t
        =
        \int_{-\infty}^{L/2}f(u)\sqrt{Q(u)}\dd u.
\]
\end{proof}
\noindent By Claim~\ref{claim:critical-change}, it is enough to prove
\[
        \int_{-\infty}^{L/2}f(u)\sqrt{Q(u)}\dd u
        \le\left(\frac14+o(1)\right)\sqrt L.
\]
Choose $A=A_n\to\infty$ such that $A=o(L)$ and $e^{-A}L^{3/2}\to0$.

\begin{claim}\label{claim:critical-initial-segment}
We have
\[
        \int_{-\infty}^{A}f(u)\sqrt{Q(u)}\dd u=o(\sqrt L).
\]
\end{claim}

\begin{proof}[Proof of Claim~\ref{claim:critical-initial-segment}]
Since $0\le N(t)\le n$ and $\int_0^12tN(t)\dd t\le1$, splitting the integral at $n^{-1/2}$ gives
\[
        \int_0^1N(t)\dd t
        \le
        \int_0^{n^{-1/2}}n\dd t
        +
        n^{1/2}\int_{n^{-1/2}}^1tN(t)\dd t
        \le
        \frac32\sqrt n.
\]
Thus $f(u)\le3/2$ for all $u$. Therefore
\[
\begin{aligned}
\int_{-\infty}^{A}f(u)\sqrt{Q(u)}\dd u
&\le C\int_{-\infty}^{A}\sqrt{Q(u)}\dd u\\
&\le C\int_{-\infty}^{0}e^u\dd u+C\int_0^A\sqrt{Q(u)}\dd u\\
&\le C+C A^{1/2}\left(\int_0^AQ(u)\dd u\right)^{1/2}\\
&=O(\sqrt A)=o(\sqrt L).
\end{aligned}
\]
Here the second line uses $Q(u)\le e^{2u}$ for $u\le0$, and the third line uses Cauchy--Schwarz together with $\int_{-\infty}^{L/2}2Q(u)\dd u\le1$.
\end{proof}
\noindent It remains to estimate the contribution from $[A,L/2]$. Define
\[
        \alpha=\int_{-\infty}^A2Q(u)\dd u,\qquad
        \alpha_-=\int_{-\infty}^02Q(u)\dd u,\qquad
        \alpha_+=\int_0^A2Q(u)\dd u.
\]

\begin{claim}\label{claim:fA-bound}
We have $f(A)\le\sqrt\alpha$.
\end{claim}

\begin{proof}[Proof of Claim~\ref{claim:fA-bound}]
Since $Q(u)\le e^{2u}$ for $u\le0$, define $q(y)$ on $(0,1]$ by $Q(u)=y^2q(y)$ with $y=e^u$. Then $0\le q(y)\le1$, and
\[
        \int_{-\infty}^0Q(u)e^{-u}\dd u=\int_0^1q(y)\dd y,\qquad
        \alpha_-=\int_{-\infty}^02Q(u)\dd u=\int_0^12yq(y)\dd y.
\]
For every measurable $q:(0,1]\to[0,1]$, the rearrangement inequality
\[
        \int_0^12yq(y)\dd y\ge\left(\int_0^1q(y)\dd y\right)^2
\]
holds, since among all functions $0\le q\le1$ with fixed $\int_0^1q(y)\dd y=m$, the integral $\int_0^1yq(y)\dd y$ is minimized by concentrating the mass on $[0,m]$. Hence
\[
        \int_{-\infty}^0Q(u)e^{-u}\dd u\le\sqrt{\alpha_-}.
\]
Since $e^{-u}\le1$ for $u\ge0$, we also have
\[
        \int_0^AQ(u)e^{-u}\dd u\le\int_0^AQ(u)\dd u=\frac{\alpha_+}{2}.
\]
Consequently
\[
        f(A)=\int_{-\infty}^AQ(u)e^{-u}\dd u
        \le\sqrt{\alpha_-}+\frac{\alpha_+}{2}
        \le\sqrt{\alpha}.
\]
The last inequality follows from $\sqrt a+b/2\le\sqrt{a+b}$ for $a,b\ge0$ with $a+b\le1$.
\end{proof}
\noindent For $u\ge A$, write $f(u)=f(A)+g(u)$, where
\[
        g(u)=\int_A^uQ(v)e^{-v}\dd v.
\]
By Claim~\ref{claim:fA-bound} and Cauchy--Schwarz,
\[
\begin{aligned}
f(A)\int_A^{L/2}\sqrt{Q(u)}\dd u
&\le \sqrt\alpha\left(L/2-A\right)^{1/2}\left(\int_A^{L/2}Q(u)\dd u\right)^{1/2}\\
&\le \sqrt\alpha\left(L/2-A\right)^{1/2}\left(\frac{1-\alpha}{2}\right)^{1/2}\\
&=\sqrt{\alpha(1-\alpha)}\left(\frac{L/2-A}{2}\right)^{1/2}\\
&\le\left(\frac14+o(1)\right)\sqrt L.
\end{aligned}
\]
Furthermore,
\[
        g(u)=\int_A^uQ(v)e^{-v}\dd v\le e^{-A}\int_A^uQ(v)\dd v.
\]
Thus
\[
\begin{aligned}
\int_A^{L/2}g(u)\sqrt{Q(u)}\dd u
&\le e^{-A}\left(\int_A^{L/2}Q(v)\dd v\right)\left(\int_A^{L/2}\sqrt{Q(u)}\dd u\right)\\
&\le C e^{-A}\sqrt L=o(\sqrt L).
\end{aligned}
\]
Combining Claim~\ref{claim:critical-initial-segment} with the two estimates on $[A,L/2]$ gives
\[
        \int_{-\infty}^{L/2}f(u)\sqrt{Q(u)}\dd u
        \le\left(\frac14+o(1)\right)\sqrt L.
\]
Since $L=\log n$, Claim~\ref{claim:critical-change} implies
\[
        \int_0^\infty F(t)\sqrt{N(t)}\dd t
        \le\left(\frac14+o(1)\right)\sqrt{n\log n}.
\]
This proves the lemma.
\end{proof}

We shall also use the following standard concentration consequence of the matrix Bernstein inequality \cite{Tropp2015} in the lower bound for the supercritical
range $3/2<p<2$.

\begin{lemma}
\label{lem:blowup-spectral-concentration}
Let $r$ be fixed, let $a=(a_1,\ldots,a_r)$ be a positive probability vector, and let $K=(K_{ij})$ be a symmetric matrix with $0\le K_{ij}\le1$. For each integer $m$, let $n_i=\lfloor a_i m\rfloor$ and let $V_1,\ldots,V_r$ be disjoint vertex sets with $|V_i|=n_i$. Let $G_m$ be the random graph obtained by joining each pair $u\in V_i$, $v\in V_j$, $u\ne v$, independently with probability $K_{ij}$. Then, with probability tending to $1$, 
\[
        \lambda(G_m)
        =
        m\lambda_{\max}(D_a^{1/2}KD_a^{1/2})+o(m).
\]
\end{lemma}

\begin{proof}
Let $A$ be the adjacency matrix of $G_m$ and let $\overline A=\mathbb E A$. Let $B$ be the block-constant matrix defined by $B_{uv}=K_{ij}$ whenever $u\in V_i$ and $v\in V_j$, including the diagonal entries. Then $\|\overline A-B\|\le1$, since $\overline A$ differs from $B$ only on the diagonal.

The matrix $B$ has rank at most $r$. On the subspace of vectors that are constant on each block, it is represented by the matrix $KD_n$, where $D_n=\diag(n_1,\ldots,n_r)$. Since $KD_n$ is similar to $D_n^{1/2}KD_n^{1/2}$, we have
$
        \lambda_{\max}(B)
        =
        \lambda_{\max}(D_n^{1/2}KD_n^{1/2}).
$
As $n_i=a_im+O(1)$ and $K$ is fixed, Weyl's inequality gives
\[
        \lambda_{\max}(B)
        =
        m\lambda_{\max}(D_a^{1/2}KD_a^{1/2})+O(1).
\]
Hence
\[
        \lambda_{\max}(\overline A)
        =
        m\lambda_{\max}(D_a^{1/2}KD_a^{1/2})+O(1).
\]

It remains to show that $\|A-\overline A\|=o(m)$ with high probability.
Write
$
        A-\overline A=\sum_{u<v}X_{uv},
        $, where $
        X_{uv}=(A_{uv}-\mathbb E A_{uv})(E_{uv}+E_{vu}).
$
The matrices $X_{uv}$ are independent, symmetric, mean-zero, and satisfy $\|X_{uv}\|\le1$. Moreover,
$
        X_{uv}^2=(A_{uv}-\mathbb E A_{uv})^2(E_{uu}+E_{vv}),
$
and therefore,
\[\begin{aligned}
        \left\|\sum_{u<v}\mathbb E X_{uv}^2\right\|
        &=\left\|\sum_{u<v}\mathbb E (A_{uv}-\mathbb E A_{uv})^2(E_{uu}+E_{vv})\right\|
        \\&=\left\|\sum_{u<v}B_{uv}(1-B_{uv})(E_{uu}+E_{vv})\right\|
        \\&=\left\|\sum_{u}\sum_{v\neq u}B_{uv}(1-B_{uv})E_{uu}\right\|
        \\&\le n/4=O(m),
        \qquad n=\sum_i n_i.
\end{aligned}\]
By the matrix Bernstein inequality for self-adjoint sums, for every $t>0$,
\[
        \mathbb P\bigl(\|A-\overline A\|\ge t\bigr)
        \le
        2n\exp\left(-\frac{t^2/2}{Cm+t/3}\right)
\]
for some constant $C=C(r,K)$. Taking $t=m^{3/4}$ gives
\[
        \mathbb P\bigl(\|A-\overline A\|\ge m^{3/4}\bigr)=o(1).
\]
Thus $\|A-\overline A\|=o(m)$ with high probability. Then by Weyl's inequality, we have 
\[
        \lambda_{\max}(A)
        =
        \lambda_{\max}(\overline A)+o(m)
        =
        m\lambda_{\max}(D_a^{1/2}KD_a^{1/2})+o(m),
\]
as required.
\end{proof}

\section{Proof of the main theorem and finite consequences}

\begin{proof}[Proof of Theorem~\ref{thm:graphon}]
If $\Delta_p(W)=0$, then taking $A=\Omega$ gives $\int_{\Omega^2}W=0$. Since $W\ge0$, this implies $W=0$ almost everywhere. Assume $\Delta_p(W)>0$. By homogeneity, replace $W$ by $W/\Delta_p(W)$ and suppose $\Delta_p(W)\le1$. Put $R=\max\{1,\|W\|_\infty\}$. Let $f\in L^2(\Omega)$ be nonnegative with $\|f\|_2=1$, define $S_t=\{x:f(x)\ge t\}$, $N(t)=\mu(S_t)$, and $F(t)=\int_0^tN(s)\dd s$. By Lemma~\ref{lem:layercake} and symmetry,
\[
\langle T_Wf,f\rangle=2\int_0^\infty\int_0^t\int_{S_t\times S_s}W\dd\mu^2\dd s\dd t.
\]
For $s\le t$ we have $S_t\subseteq S_s$, so Lemma~\ref{lem:nested-incidence} gives
\[
\int_{S_t\times S_s}W\le C_pN(s)\min\{RN(t),N(t)^{p-1}\}=C_pN(s)\phi_R(N(t)).
\]
Consequently,
\[
\langle T_Wf,f\rangle\le C_p\int_0^\infty F(t)\phi_R(N(t))\dd t.
\]
Since $\int_0^\infty2tN(t)\dd t=\|f\|_2^2=1$, Lemma~\ref{lem:hardy} gives $\langle T_Wf,f\rangle\le C_p\Phi_p(R)$. Taking the supremum over nonnegative unit vectors in $L^2(\Omega)$ and then rescaling by $\Delta_p(W)$ proves the theorem, because $\max\{1,\|W\|_\infty/\Delta_p(W)\}\le1+\|W\|_\infty/\Delta_p(W)$ and $\Phi_p$ is nondecreasing.
\end{proof}

\begin{corollary}\label{cor:Cp-finite}
For every $3/2<p<2$, one has $0<\mathfrak C_p<\infty$.
\end{corollary}

\begin{proof}
Positivity follows by taking a one-atom kernel. We next check that the denominator in Definition~\ref{def:Cp} is a positive maximum for every admissible pair $(a,K)$. Since $K\ne0$ and $K$ is nonnegative, there exists some $s$ with $0\le s_i\le a_i$ and $s^{\mathsf T}Ks>0$. Since $2-p>0$, if $0<s_i<a_i$ for every nonzero coordinate of $s$, then replacing $s$ by $qs$ for some $q>1$ with $qs_i\le a_i$ increases the value of $s^{\mathsf T}Ks/(\sum_i s_i)^p$ by the factor $q^{2-p}$. Hence the supremum is attained on the compact set where at least one active coordinate reaches its upper bound. Equivalently, one may write
\[
        \max_{\substack{0\le s_i\le a_i\\ \sum_i s_i>0}}
        \frac{s^{\mathsf T}Ks}{(\sum_i s_i)^p}
        =
        \max_{\substack{0\le s_i\le a_i\\ \sum_i s_i\ge \min_i a_i}}
        \frac{s^{\mathsf T}Ks}{(\sum_i s_i)^p}>0.
\]
For finiteness, let $(a,K)$ be any finite kernel and realize it as a step kernel $W$ on $[0,1]$ with atom intervals of lengths $a_i$ and values $K_{ij}$. Then $$\|T_W\|_{2\to2}=\lambda_{\max}(KD_a)=\lambda_{\max}(D_a^{1/2}KD_a^{1/2})$$
and $\Delta_p(W)=\max_{0\le s_i\le a_i,\ \sum_i s_i>0}s^{\mathsf T}Ks/(\sum_i s_i)^p$. Since $p>3/2$, Theorem~\ref{thm:graphon} gives $\|T_W\|_{2\to2}\le C_p\Delta_p(W)$, and therefore every quotient in Definition~\ref{def:Cp} is at most $2C_p$.
\end{proof}

\begin{proof}[Proof of Theorem~\ref{thm:main} for $1<p<3/2$]
The lower bound is given by the star. For $G=K_{1,n-1}$, one has $\lambda(G)=\sqrt{n-1}$, and every nonempty induced subgraph with edges is a star with $t$ leaves. Hence
\[
d_p(K_{1,n-1})=\max_{1\le t\le n-1}\frac{t}{(t+1)^p}\to M_p,\qquad M_p=\max_{t\in\mathbb N_{\ge1}}\frac{t}{(t+1)^p}.
\]
The continuous function $t/(t+1)^p$ has derivative $(t+1)^{-p-1}(1-(p-1)t)$, so its continuous maximum is at $t=1/(p-1)$. It follows that the integer maximum is attained by one of $\floor{1/(p-1)}$ and $\ceil{1/(p-1)}$.

We prove the matching upper bound. Let $G$ be an $n$-vertex graph with $e(G)>0$, set $D=d_p(G)$, and let $x\ge0$ be a Perron unit vector of $G$. Fix $\eps>0$ and put $H=\{v:x_v>\eps\}$ and $T=V(G)\setminus H$. Then $|H|\le\eps^{-2}$, and
\[
\lambda(G)=x_H^{\mathsf T}A_Hx_H+2x_H^{\mathsf T}A_{H,T}x_T+x_T^{\mathsf T}A_Tx_T.
\]
The first term is $O_\eps(1)$, and after division by $D\sqrt n$ it tends to $0$, since $G$ contains an induced edge and therefore $D=d_p(G)\ge 2^{-p}$. By Lemma~\ref{lem:tail}, the third term is at most $C_pD(\eps^{3-2p}\sqrt n+n^{p-1})$, so its contribution to $\lambda(G)/(D\sqrt n)$ is at most $C_p\eps^{3-2p}+o(1)$.

It remains to estimate the middle term. Let $B=A_{H,T}$. Classify vertices $v\in T$ according to their neighborhood $R(v)=N(v)\cap H$; that is, for each $Q\subseteq H$, a vertex $v\in T$ is said to be of type $Q$ if $R(v)=Q$. Let $t_*$ be an integer for which $M_p=t_*/(t_*+1)^p$. A type $R$ is called rare if fewer than $|R|t_*$ vertices of $T$ have neighborhood $R$. 
The total number of rare vertices is at most
\[
        t_*\sum_{R\subseteq H}|R|=t_*|H|2^{|H|-1}=O_{\eps,p}(1).
\]
Let $B_{\rm rare}$ be the submatrix of $B$ formed by the rare columns. Since each such column has at most $|H|$ ones, and since $|H|\le \eps^{-2}$, we have
\[
        \|B_{\rm rare}\|_F^2
        \le |H|\cdot(\text{number of rare vertices})
        =O_{\eps,p}(1).
\]
Therefore $\|B_{\rm rare}\|\le \|B_{\rm rare}\|_F=O_{\eps,p}(1)$.

By the triangle inequality and the bound on $\|B_{\rm rare}\|$, discarding these rare columns changes $\|B\|$ by only $O_{\eps,p}(1)$. Let $B_{\text{remain}}$ denote the matrix obtained from $B$ after discarding the rare columns. For the remaining types, let $q_*=\max |R|$. If $q_*>0$, choose a nonrare type $R$ with $|R|=q_*$. Since this type is nonrare, there are at least $q_*t_*$ vertices $v\in T$ with $N(v)\cap H=R$. The induced subgraph on $R$ together with any $q_*t_*$ of these vertices contains $K_{q_*,q_*t_*}$ as a subgraph, and hence has $q_*(t_*+1)$ vertices and at least $q_*^2t_*$ edges.
Therefore $D\ge q_*^{2-p}M_p$. On the other hand, for every unit vector $u\in\mathbb R^H$,
\[
u^{\mathsf T}B_{\text{remain}}B_{\text{remain}}^{\mathsf T}u=\sum_{v\in T}\left(\sum_{i\in R(v)}u_i\right)^2\le\sum_{v\in T}|R(v)|\sum_{i\in R(v)}u_i^2\le q_*n.
\]
Thus $\|B\|\le\|B_{\text{remain}}\|+\|B_{\text{rare}}\|=\sqrt{q_*n}+O_{\eps,p}(1)$, and
\[
2x_H^{\mathsf T}Bx_T\le2\|x_H\|_2\|x_T\|_2\|B\|\le\|B\|\le\frac{D}{M_p}q_*^{p-3/2}\sqrt n+O_{\eps,p}(1)\le\frac{D}{M_p}\sqrt n+O_{\eps,p}(1),
\]
because $p<3/2$. We have shown
\[
\lambda(G)\le\frac{D}{M_p}\sqrt n+C_pD\eps^{3-2p}\sqrt n+o(D\sqrt n).
\]
Letting $n\to\infty$ and then $\eps\to0$ gives $\limsup_{n\to\infty}\Lambda_p(n)/\sqrt n\le1/M_p$. Together with the star lower bound, this proves the assertion for $1<p<3/2$.
\end{proof}

\begin{proof}[Proof of Theorem~\ref{thm:main} for $p=3/2$]
We first prove the upper bound. Let $G$ be an $n$-vertex graph, let $D=d_{3/2}(G)$, and let $x\ge0$ be a Perron unit vector. For $t>0$, set $S_t=\{v:x_v\ge t\}$, $N(t)=|S_t|$ and $F(t)=\int_0^tN(s)\dd s$. If $A,C$ are disjoint, $|A|=a$ and $|C|=c$, then
\[
e(A,C)\le\frac{3\sqrt3}{2}D(a+c)\sqrt a.
\]
Indeed, if $c\ge2a$, choose uniformly a subset $T\subseteq C$ of size $2a$. Then $(2a/c)e(A,C)=\mathbb E e(A,T)\le D(3a)^{3/2}$, which gives $e(A,C)\le(3\sqrt3/2)Dc\sqrt a$. If $c<2a$, then $e(A,C)\le e(A\cup C)\le D(a+c)^{3/2}\le(3\sqrt3/2)D(a+c)\sqrt a$.

The layer-cake formula for $x^{\mathsf T}Ax$ gives
\[
\lambda(G)=2\int_0^\infty\int_0^t\cE(S_t,S_s)\dd s\dd t.
\]
For $s\le t$, the preceding incidence estimate and $e(S_t)\le D N(t)^{3/2}$ imply
\[
\cE(S_t,S_s)=2e(S_t)+e(S_t,S_s\setminus S_t)\le 2D N(t)^{3/2}+\frac{3\sqrt3}{2}DN(s)\sqrt{N(t)}.
\]
Consequently,
\[
\lambda(G)\le4D\int_0^\infty tN(t)^{3/2}\dd t+3\sqrt3D\int_0^\infty F(t)\sqrt{N(t)}\dd t.
\]
Since $N(t)\le\min(n,t^{-2})$ and since $\|x\|_2=1$ implies $x_v\le1$ for every vertex, we have $N(t)=0$ for $t>1$. Therefore, 
$$\begin{aligned}\int_0^\infty tN(t)^{3/2}\dd t&=\int_0^{n^{-1/2}} tN(t)^{3/2}\dd t+\int_{n^{-1/2}}^1 tN(t)^{3/2}\dd t
\\&\leq\int_0^{n^{-1/2}} tn^{3/2}\dd t+\int_{n^{-1/2}}^1 t^{-2}\dd t
\\&=\frac{\sqrt{n}}{2}+\sqrt n-1=O(\sqrt n)\end{aligned}$$
Moreover, Lemma~\ref{lem:critical-hardy} bounds the second integral by $(1/4+o(1))\sqrt{n\log n}$. Thus
\[
\lambda(G)\le\left(\frac{3\sqrt3}{4}+o(1)\right)D\sqrt{n\log n}.
\]

We now prove the matching lower bound. Fix $h,\beta,\eta>0$ and let $N$ tend to infinity. Let $B$ be a set of $N$ vertices and let $L=\floor{(1-\eta)\log N/h}$. For $0\le i<L$, let $H_i$ be disjoint sets with $m_i=|H_i|=\floor{\beta N(1-e^{-h})e^{-ih}}$. Then $\sum_i m_i=(\beta+o(1))N$ and $m_{L-1}=N^{\eta+o(1)}$. Choose a sequence $c_N$ such that $c_N/\log N\to\infty$ and $c_N^2h=o(m_{L-1})$; for instance $c_N=(\log N)^2$ works for fixed $h,\beta,\eta$. Between $B$ and $H_i$, put independent edges with probability $q_i=c_N\sqrt h/\sqrt{m_i}$, and put no other edges. The condition $c_N^2h=o(m_{L-1})$ ensures that $q_i\le1$ for all large $N$. Let $W_2=Lh=(1-\eta+o(1))\log N$. 
For each $0\le i<L$, the random variable $e(B,H_i)$ is binomial with mean value
\[
        \mu_i=Nm_iq_i=c_NN\sqrt h\,\sqrt{m_i}.
\]
Since $m_i\ge m_{L-1}$, we have $\mu_i\ge\mu_*:=c_NN\sqrt h\,\sqrt{m_{L-1}}$. Choose $\delta=(\log N)^{-1}$. By Chernoff's inequality,
\[
\Pr\left(\left|e(B,H_i)-\mu_i\right|>\delta\mu_i\right)
\le
2\exp\left(-\frac{\delta^2\mu_i}{3}\right)
\le
2\exp\left(-\frac{\delta^2\mu_*}{3}\right).
\]
Since $m_{L-1}=N^{\eta+o(1)}$, $c_N/\log N\to\infty$ and $L=O(\log N)$, it follows that
$
        L\exp\left(-\frac{\delta^2\mu_*}{3}\right)=o(1).
$
Therefore, by the union bound, with probability $1-o(1)$ the estimates
$
        |e(B,H_i)-\mu_i|\le\delta\mu_i
$
hold simultaneously for all $0\le i<L$. Since $\delta=o(1)$, we obtain, uniformly for all $0\le i<L$,
\[
        e(B,H_i)=(1+o(1))Nm_iq_i=(1+o(1))c_NN\sqrt h\,\sqrt{m_i}.
\]
On this event, define $z_b=(2N)^{-1/2}$ for $b\in B$ and $z_v=\sqrt h/(2W_2m_i)^{1/2}$ for $v\in H_i$. Then $\|z\|_2=1$ and
\[
z^{\mathsf T}Az=(1+o(1))c_N\sqrt{NW_2}\ge(1-o(1))c_N\sqrt{(1-\eta)N\log N}.
\]
It remains to bound $d_{3/2}$ uniformly over all vertex subsets. Let $S=X\cup\bigcup_iY_i$ with $X\subseteq B$ and $Y_i\subseteq H_i$. Put $x=|X|$, $y_i=|Y_i|$, $Y=\sum_i y_i$ and $s=x+Y=|S|$. The expectation $\mu(S)=\mathbb E e(S)$ is
\[
\mu(S)=c_Nx\sum_i\frac{\sqrt h\,y_i}{\sqrt{m_i}}.
\]
Define
\[
T_{h,N}=\sup_{\substack{0\le y_i\le m_i\\ \sum_i y_i>0}}\frac{\sum_i\sqrt h\,y_i/\sqrt{m_i}}{\sqrt{\sum_i y_i}}.
\]
We justify the evaluation of $T_{h,N}$. Let $a_i=\sqrt h/\sqrt{m_i}$. Then $a_i$ is increasing in $i$. For a fixed value of $Y=\sum_i y_i$, the linear functional $\sum_i a_iy_i$ is maximized by filling the layers with largest $a_i$, equivalently by taking full layers from the tail $i=L-1,L-2,\ldots$, with at most one partially filled next layer. A partially filled layer cannot give a larger value than one of the two adjacent full-layer choices, since $(A+\theta c)/(B+\theta d)^{1/2}$ has no interior maximum for $0\le\theta\le1$. Thus it is enough to consider $k$ full tail layers. Since $m_i=(1+o(1))\beta N(1-e^{-h})e^{-ih}$ up to a multiplicative $1+o(1)$ error, uniformly because $m_{L-1}\to\infty$,
for such a choice,
\[
\frac{\sum_i\sqrt h\,y_i/\sqrt{m_i}}{\sqrt{\sum_i y_i}}=(1+o(1))\sqrt h\,\frac{1-e^{-kh/2}}{1-e^{-h/2}}\left(\frac{1-e^{-h}}{1-e^{-kh}}\right)^{1/2}\le(1+o(1))\sqrt h\,\frac{\sqrt{1-e^{-h}}}{1-e^{-h/2}}.
\]
The last inequality follows from
\[
        \frac{1-e^{-kh/2}}{\sqrt{1-e^{-kh}}}
        =
        \left(\frac{1-e^{-kh/2}}{1+e^{-kh/2}}\right)^{1/2}
        \le 1.
\]
Since this factor is $1+o_h(1)$ as $k\to\infty$, taking $k\to\infty$ with $k\le L$ gives the reverse inequality up to $1+o(1)$, and therefore
\[
T_{h,N}=(1+o(1))T_h,\qquad T_h=\sqrt h\,\frac{\sqrt{1-e^{-h}}}{1-e^{-h/2}},
\]
as $N\to\infty$ with $h$ fixed.
Hence
\[
\mu(S)\le(1+o(1))c_NT_hx\sqrt Y\le(1+o(1))c_NT_h\frac{2}{3\sqrt3}s^{3/2},
\]
where the last inequality is the sharp maximum of $x\sqrt Y/(x+Y)^{3/2}$. Fix $\gamma>0$ and put $R_s=(1+\gamma)c_NT_h(2/(3\sqrt3))s^{3/2}$. For all large $N$, $R_s\ge(1+\gamma/2)\mu(S)$ for every $S$ of size $s$. The random variable $e(S)$ is a sum of independent Bernoulli variables, so Chernoff's inequality gives $\mathbb P(e(S)\ge R_s)\le\exp(-c_\gamma c_Ns^{3/2})$. 
The total number of vertex subsets of size $s$ is at most $(C_{\beta,h}N)^s$. Since $c_N/\log N\to\infty$, for all sufficiently large $N$ we have
$
        \log(C_{\beta,h}N)\le \frac{c_\gamma}{2}c_N .
$
Therefore, by the union bound and the estimate $s^{3/2}\ge s$ for $s\ge1$,
\[
\begin{aligned}
\mathbb P\left(\exists S:e(S)>(1+\gamma)c_NT_h\frac{2}{3\sqrt3}|S|^{3/2}\right)
&\le \sum_{s\ge1} (C_{\beta,h}N)^s\exp\left(-c_\gamma c_Ns^{3/2}\right)\\
&= \sum_{s\ge1}\exp\left(s\log(C_{\beta,h}N)-c_\gamma c_Ns^{3/2}\right)\\
&\le \sum_{s\ge1}\exp\left(-\frac{c_\gamma}{2}c_Ns\right)=o(1).
\end{aligned}
\]
Intersecting this event with the earlier simultaneous concentration of the layer edge counts still gives probability $1-o(1)$. Since $\gamma>0$ is arbitrary, with probability $1-o(1)$,
\[
e(S)\le(1+o(1))c_N\left(\sqrt h\,\frac{\sqrt{1-e^{-h}}}{1-e^{-h/2}}\right)\frac{2}{3\sqrt3}|S|^{3/2}
\]
for every $S$. Hence
\[
d_{3/2}(G)\le(1+o(1))c_N\left(\sqrt h\,\frac{\sqrt{1-e^{-h}}}{1-e^{-h/2}}\right)\frac{2}{3\sqrt3}.
\]
The total number of vertices is $n=(1+\beta+o(1))N$. Therefore
\[
\frac{\lambda(G)}{d_{3/2}(G)\sqrt{n\log n}}\ge\frac{z^{\mathsf T}Az}{d_{3/2}(G)\sqrt{n\log n}}\ge (1-o(1))\frac{\sqrt{1-\eta}}{\sqrt{1+\beta}}\frac{3\sqrt3}{2}\left(\sqrt h\,\frac{\sqrt{1-e^{-h}}}{1-e^{-h/2}}\right)^{-1}.
\]
Letting $N\to\infty$ and then $\eta,\beta,h\to0$ gives the lower bound $3\sqrt3/4$.
\end{proof}

\begin{proof}[Proof of Theorem~\ref{thm:main} for $3/2<p<2$]

We prove first that
$
        \liminf_{n\to\infty}\frac{\Lambda_p(n)}{n^{p-1}} \ge \mathfrak C_p.
$
Fix $r,a,K$ as in Definition~\ref{def:Cp}. By homogeneity we may assume $0\le K_{ij}\le1$. For a large integer $m$, let $G_m$ be the random blow-up associated with $(a,K)$ as in Lemma~\ref{lem:blowup-spectral-concentration}. Thus, with probability tending to $1$,
\[
        \lambda(G_m)
        =
        m\lambda_{\max}(D_a^{1/2}KD_a^{1/2})+o(m).
\]

We next provide the uniform density estimate. For any subset $S$, write $s_i=|S\cap V_i|/m$ and $\sigma=\sum_i s_i$. Fix $\tau>0$. Hoeffding's inequality with an additive error $\varepsilon m^2$, followed by a union bound over all subsets, gives uniformly for all $S$ with $|S|\ge\tau m$ that
\[
2e(S)=m^2s^{\mathsf T}Ks+o(m^2).
\]
Consequently, uniformly in this range,
\[
\frac{e(S)}{|S|^p}=\frac12m^{2-p}\frac{s^{\mathsf T}Ks}{\sigma^p}+o(m^{2-p}).
\]
For $|S|<\tau m$ one has the deterministic bound $e(S)/|S|^p\le |S|^{2-p}/2\le \tau^{2-p}m^{2-p}/2$. Taking first $m\to\infty$ and then $\tau\downarrow0$, and using subsets with prescribed proportions to get the matching lower bound, yields
\[
d_p(G_m)=\frac12m^{2-p}\max_{\substack{0\le s_i\le a_i\\ \sum_i s_i>0}}\frac{s^{\mathsf T}Ks}{(\sum_i s_i)^p}+o(m^{2-p})
\]
with probability tending to $1$. Hence
\[
\frac{\lambda(G_m)}{d_p(G_m)|V(G_m)|^{p-1}}\to\frac{2\lambda_{\max}(D_a^{1/2}KD_a^{1/2})}{\max_{0\le s_i\le a_i,\ \sum_i s_i>0}\dfrac{s^{\mathsf T}Ks}{(\sum_i s_i)^p}}.
\]
Taking the supremum over $r,a,K$ proves the lower bound.

We now prove the upper bound. Let $G=G_n$ be an $n$-vertex graph, set $D=d_p(G)$ and let $x\ge0$ be a Perron unit vector. If $\lambda(G)=o(Dn^{p-1})$, then the desired estimate is immediate. Otherwise, along a subsequence, $\lambda(G)\ge cDn^{p-1}$ for some fixed $c>0$. Since $\lambda x_v=(Ax)_v\le\sqrt{d(v)}\le\sqrt n$ and $D\ge2^{-p}$, we have $\|x\|_\infty\le Cn^{3/2-p}=o(1)$. Choose $M=M_n\to\infty$ with $M\|x\|_\infty^2\to0$ and $M=o(n)$. Order the vertices so that $x_1\ge x_2\ge\cdots\ge x_n$, and partition them into consecutive blocks $P_1,\ldots,P_r$ of sizes between $M$ and $2M$, except possibly the last block, which is joined to the preceding block if necessary. Let $y$ be the vector equal to the average of $x$ on each block. By Lemma~\ref{lem:block-averaging}, $\|x-y\|_2=o(1)$. Since $\|A(G)\|=\lambda(G)$, we have $$|x^{\mathsf T}Ax-y^{\mathsf T}Ay|=|x^{\mathsf T}A(x-y)+(x-y)^{\mathsf T}Ay|\le\lambda(G)\|x-y\|_2(\|x\|_2+\|y\|_2)=o(\lambda(G)),$$ and hence $y^{\mathsf T}Ay=(1-o(1))\lambda(G)$ and $\|y\|_2=1+o(1)$.

Set $a_i=|P_i|/n$ and define $K=(K_{ij})$ by $K_{ij}=e(P_i,P_j)/(|P_i||P_j|)$ for $i\ne j$ and $K_{ii}=2e(P_i)/|P_i|^2$. If $y$ takes the value $\xi_i$ on $P_i$ and $u_i=\sqrt{na_i}\xi_i$, then $\sum_i u_i^2=\|y\|_2^2$ and
\[
\frac{y^{\mathsf T}Ay}{n}=u^{\mathsf T}D_a^{1/2}KD_a^{1/2}u.
\]
Therefore
\[
\frac{\lambda(G)}{n}=(1+o(1))\frac{y^{\mathsf T}Ay}{n\|y\|_2^2}=(1+o(1))\frac{u^{\mathsf T}D_a^{1/2}KD_a^{1/2}u}{\|u\|_2^2}\leq(1+o(1))\lambda_{\max}(D_a^{1/2}KD_a^{1/2}).
\]
We next compare the density functional with $D$. Let $0\le s_i\le a_i$ and $\sigma=\sum_i s_i>0$. Since the quotient $s^{\mathsf T}Ks/\sigma^p$ is homogeneous of degree $2-p>0$, it suffices to consider a maximizer for which at least one coordinate is saturated. For such a maximizer, $\sigma n\ge M$. In each block $P_i$, choose vertices independently with probability $\theta_i=s_i/a_i$, and let $R$ be the resulting random set. Then $\mathbb E|R|=\sigma n\ge M$ and, by the definition of $K$,
\[
\mathbb E\,2e(R)=n^2s^{\mathsf T}Ks
\]
exactly. Since $2e(R)\le2D|R|^p$ for every outcome and $1<p<2$, one has
\[
\mathbb E|R|^p\le(\mathbb E|R|^2)^{p/2}\le\left((\sigma n)^2+\sigma n\right)^{p/2}=(1+o(1))(\sigma n)^p,
\]
uniformly for $\sigma n\ge M$. Hence
\[
s^{\mathsf T}Ks\le(2+o(1))Dn^{p-2}\sigma^p.
\]
Taking the maximum over $s$ gives
\[
\max_{\substack{0\le s_i\le a_i\\ \sum_i s_i>0}}\frac{s^{\mathsf T}Ks}{(\sum_i s_i)^p}\le(2+o(1))Dn^{p-2}.
\]
Combining this estimate with the spectral bound yields
\[
\frac{\lambda(G)}{Dn^{p-1}}\le(1+o(1))\frac{2\lambda_{\max}(D_a^{1/2}KD_a^{1/2})}{\max_{0\le s_i\le a_i,\ \sum_i s_i>0}\dfrac{s^{\mathsf T}Ks}{(\sum_i s_i)^p}}\le(1+o(1))\mathfrak C_p.
\]
Taking the limsup proves $\limsup\Lambda_p(n)/n^{p-1}\le\mathfrak C_p$.
\end{proof}

\begin{proof}[Proof of Theorem~\ref{thm:main} for $p=2$]
Let $D=d_2(G)$ and let $\omega=\omega(G)$ be the clique number. The clique $K_\omega$ gives $D\ge\binom{\omega}{2}/\omega^2=(1-1/\omega)/2$. Wilf's inequality \cite{Wilf1986} states that $\lambda(G)\le(1-1/\omega(G))n$. Hence $\lambda(G)\le2Dn$. For $G=K_n$, one has $d_2(K_n)=\binom n2/n^2=(n-1)/(2n)$ and $\lambda(K_n)=n-1=2d_2(K_n)n$, so $\Lambda_2(n)=2n$.
\end{proof}

\begin{proof}[Proof of Corollary~\ref{cor:finite-orders}]
This is Theorem~\ref{thm:main} rewritten under the hypothesis $d_p(G)\le D$. The endpoint $p=2$ is finite-$n$ exact, while the other three regimes are asymptotic and the constants are attained by the constructions used in the proof of Theorem~\ref{thm:main}.
\end{proof}

\section{Concluding remarks}

Theorem~\ref{thm:main} resolves Guiduli's power-law hereditary density problem for all $1<p\le2$ in sharp asymptotic form. The only leading constant not given by an elementary expression is the supercritical constant $\mathfrak C_p$, which is characterized by an exact variational formula over finite kernels. It would be natural to determine whether this variational problem admits a closed-form solution, or whether its near-extremizers have a canonical finite-kernel or graph-limit structure.

The proof also shows that the transition at $p=3/2$ is not an artefact of the method. Below the transition point, the extremal examples are essentially star-like; at the transition point, a multiscale construction is needed; above the transition point, blow-ups of finite kernels determine the leading constant. Finding other hereditary spectral problems with the same trichotomy seems to be an interesting direction for future work.

\section*{Acknowledgment}
The work is partly  supported by National Natural Science Foundation of China (Nos. 12371354, W2521102), the Science and Technology Commission of Shanghai Municipality (No.25LN3200600) and the Montenegrin-Chinese Science and Technology Cooperation Project (No. 4-3).

\bibliographystyle{plain}
\bibliography{0601Guiduli}

@phdthesis{Guiduli1996,
  author = {Guiduli, Barry},
  title = {Spectral Extrema for Graphs},
  school = {University of Chicago},
  year = {1996}
}

@article{BabaiGuiduli2009,
  author = {Babai, L{\'a}szl{\'o} and Guiduli, Barry},
  title = {Spectral Extrema for Graphs: The {Z}arankiewicz Problem},
  journal = {Electronic Journal of Combinatorics},
  volume = {16},
  number = {1},
  pages = {Research Paper R123, 8 pp.},
  year = {2009},
  doi = {10.37236/212}
}

@article{LiWangZhai2024,
  author = {Li, Rui and Wang, Anyao and Zhai, Mingqing},
  title = {Spectral Radius of Graphs with Size Constraints: Resolving a Conjecture of {G}uiduli},
  journal = {arXiv preprint},
  year = {2024},
  eprint = {2412.06375},
  archivePrefix = {arXiv},
  primaryClass = {math.CO},
  note = {arXiv:2412.06375}
}

@article{LiuNing2023,
  author = {Liu, Lele and Ning, Bo},
  title = {Unsolved Problems in Spectral Graph Theory},
  journal = {Operations Research Transactions},
  volume = {27},
  number = {4},
  pages = {33--60},
  year = {2023},
  note = {Also available as arXiv:2305.10290}
}

@article{LiuNing2026,
  author = {Liu, Lele and Ning, Bo},
  title = {On Spectral {T}ur{\'a}n Theorems: Confirming a Conjecture of {G}uiduli and Two Problems of {N}ikiforov},
  journal = {arXiv preprint},
  year = {2026},
  eprint = {2605.05048},
  archivePrefix = {arXiv},
  primaryClass = {math.CO},
  note = {arXiv:2605.05048}
}

@article{Wilf1986,
  author = {Wilf, Herbert S.},
  title = {Spectral Bounds for the Clique and Independence Numbers of Graphs},
  journal = {Journal of Combinatorial Theory, Series B},
  volume = {40},
  number = {1},
  pages = {113--117},
  year = {1986},
  doi = {10.1016/0095-8956(86)90069-9}
}

@article{LovaszSzegedy2006,
  author = {Lov{\'a}sz, L{\'a}szl{\'o} and Szegedy, Bal{\'a}zs},
  title = {Limits of Dense Graph Sequences},
  journal = {Journal of Combinatorial Theory, Series B},
  volume = {96},
  number = {6},
  pages = {933--957},
  year = {2006},
  doi = {10.1016/j.jctb.2006.05.002}
}

@article{BorgsChayesCohnZhao2019,
  author = {Borgs, Christian and Chayes, Jennifer T. and Cohn, Henry and Zhao, Yufei},
  title = {An {$L^p$} Theory of Sparse Graph Convergence {I}: Limits, Sparse Random Graph Models, and Power Law Distributions},
  journal = {Transactions of the American Mathematical Society},
  volume = {372},
  number = {5},
  pages = {3019--3062},
  year = {2019},
  doi = {10.1090/tran/7543}
}

@article{BorgsChayesCohnZhao2018,
  author = {Borgs, Christian and Chayes, Jennifer T. and Cohn, Henry and Zhao, Yufei},
  title = {An {$L^p$} Theory of Sparse Graph Convergence {II}: {LD} Convergence, Quotients, and Right Convergence},
  journal = {Annals of Probability},
  volume = {46},
  number = {1},
  pages = {337--396},
  year = {2018},
  doi = {10.1214/17-AOP1187}
}

@article{Razborov2007,
  author = {Razborov, Alexander A.},
  title = {Flag Algebras},
  journal = {Journal of Symbolic Logic},
  volume = {72},
  number = {4},
  pages = {1239--1282},
  year = {2007},
  doi = {10.2178/jsl/1203350785}
}

@article{Razborov2008,
  author = {Razborov, Alexander A.},
  title = {On the Minimal Density of Triangles in Graphs},
  journal = {Combinatorics, Probability and Computing},
  volume = {17},
  number = {4},
  pages = {603--618},
  year = {2008},
  doi = {10.1017/S0963548308009085}
}

@article{Reiher2016,
  author = {Reiher, Christian},
  title = {The Clique Density Theorem},
  journal = {Annals of Mathematics},
  volume = {184},
  number = {3},
  pages = {683--707},
  year = {2016},
  doi = {10.4007/annals.2016.184.3.1}
}

@article{HatamiNorine2011,
  author = {Hatami, Hamed and Norine, Serguei},
  title = {Undecidability of Linear Inequalities in Graph Homomorphism Densities},
  journal = {Journal of the American Mathematical Society},
  volume = {24},
  number = {2},
  pages = {547--565},
  year = {2011},
  doi = {10.1090/S0894-0347-2010-00687-X}
}

@book{Lovasz2012,
  author = {Lov{\'a}sz, L{\'a}szl{\'o}},
  title = {Large Networks and Graph Limits},
  series = {American Mathematical Society Colloquium Publications},
  volume = {60},
  publisher = {American Mathematical Society},
  address = {Providence, RI},
  year = {2012},
  isbn = {978-0-8218-9085-1}
}

@article{Szegedy2011,
  author = {Szegedy, Bal{\'a}zs},
  title = {Limits of Kernel Operators and the Spectral Regularity Lemma},
  journal = {European Journal of Combinatorics},
  volume = {32},
  number = {7},
  pages = {1156--1167},
  year = {2011},
  doi = {10.1016/j.ejc.2011.03.005}
}

@article{BollobasBorgsChayesRiordan2010,
  author = {Bollob{\'a}s, B{\'e}la and Borgs, Christian and Chayes, Jennifer and Riordan, Oliver},
  title = {Percolation on Dense Graph Sequences},
  journal = {Annals of Probability},
  volume = {38},
  number = {1},
  pages = {150--183},
  year = {2010},
  doi = {10.1214/09-AOP478}
}

@article{BollobasJansonRiordan2007,
  author = {Bollob{\'a}s, B{\'e}la and Janson, Svante and Riordan, Oliver},
  title = {The Phase Transition in Inhomogeneous Random Graphs},
  journal = {Random Structures \& Algorithms},
  volume = {31},
  number = {1},
  pages = {3--122},
  year = {2007},
  doi = {10.1002/rsa.20168}
}

@article{Nikiforov2007,
  author = {Nikiforov, Vladimir},
  title = {Bounds on Graph Eigenvalues {II}},
  journal = {Linear Algebra and its Applications},
  volume = {427},
  number = {2--3},
  pages = {183--189},
  year = {2007},
  doi = {10.1016/j.laa.2007.07.010}
}

@article{Nikiforov2010PathsCycles,
  author = {Nikiforov, Vladimir},
  title = {The Spectral Radius of Graphs without Paths and Cycles of Specified Length},
  journal = {Linear Algebra and its Applications},
  volume = {432},
  number = {9},
  pages = {2243--2256},
  year = {2010},
  doi = {10.1016/j.laa.2009.05.023}
}

@mastersthesis{Nosal1970,
  author = {Nosal, Eva},
  title = {Eigenvalues of Graphs},
  school = {University of Calgary},
  year = {1970}
}

@article{Stanley1987,
  author = {Stanley, Richard P.},
  title = {A Bound on the Spectral Radius of Graphs with $e$ Edges},
  journal = {Linear Algebra and its Applications},
  volume = {87},
  pages = {267--269},
  year = {1987},
  doi = {10.1016/0024-3795(87)90172-8}
}

@article{Mantel1907,
  author  = {Mantel, W.},
  title   = {Problem 28},
  journal = {Wiskundige Opgaven met de Oplossingen},
  volume  = {10},
  pages   = {60--61},
  year    = {1907}
}

@article{Turan1941,
  author  = {Tur{\'a}n, Paul},
  title   = {On an extremal problem in graph theory},
  journal = {Mat. Fiz. Lapok},
  volume  = {48},
  pages   = {436--452},
  year    = {1941}
}

@article{TaitTobin2017,
  author = {Tait, Michael and Tobin, Josh},
  title = {Three Conjectures in Extremal Spectral Graph Theory},
  journal = {Journal of Combinatorial Theory, Series B},
  volume = {126},
  pages = {137--161},
  year = {2017},
  doi = {10.1016/j.jctb.2017.04.006}
}

@article{CioabaDesaiTait2024,
  author = {Cioab{\u a}, Sebastian M. and Desai, Dheer Noal and Tait, Michael},
  title = {The Spectral Even Cycle Problem},
  journal = {Combinatorial Theory},
  volume = {4},
  number = {1},
  pages = {Paper No. 10},
  year = {2024},
  eprint = {2205.00990},
  archivePrefix = {arXiv},
  primaryClass = {math.CO}
}

@article{WangKangXue2023,
  author = {Wang, Jing and Kang, Liying and Xue, Yisai},
  title = {On a Conjecture of Spectral Extremal Problems},
  journal = {Journal of Combinatorial Theory, Series B},
  volume = {159},
  pages = {20--41},
  year = {2023},
  doi = {10.1016/j.jctb.2022.11.002}
}

@article{ByrneDesaiTait2024,
  author = {Byrne, John and Desai, Dheer Noal and Tait, Michael},
  title = {A General Theorem in Spectral Extremal Graph Theory},
  journal = {Transactions of the American Mathematical Society},
  note = {To appear; also available as arXiv:2401.07266},
  year = {2025},
  eprint = {2401.07266},
  archivePrefix = {arXiv},
  primaryClass = {math.CO}
}

@article{Liu2024GraphLimitsSpectral,
  author = {Liu, Lele},
  title = {Graph Limits and Spectral Extremal Problems for Graphs},
  journal = {SIAM Journal on Discrete Mathematics},
  volume = {38},
  number = {1},
  pages = {590--608},
  year = {2024},
  doi = {10.1137/22M1508807}
}

@article{BreenRiasanovskyTaitUrschel2022,
  author = {Breen, Jane and Riasanovsky, Alex W. N. and Tait, Michael and Urschel, John},
  title = {Maximum Spread of Graphs and Bipartite Graphs},
  journal = {Communications of the American Mathematical Society},
  volume = {2},
  pages = {417--480},
  year = {2022},
  doi = {10.1090/cams/14}
}

@article{KumarLiuMonterdePragadaTait2026,
  author = {Kumar, Hitesh and Liu, Lele and Monterde, Hermie and Pragada, Shivaramakrishna and Tait, Michael},
  title = {Maximum Spectral Sum of Graphs},
  journal = {arXiv preprint},
  year = {2026},
  eprint = {2604.00512},
  archivePrefix = {arXiv},
  primaryClass = {math.CO}
}

@article{VizueteGarinFrasca2021,
  author = {Vizuete, Renato and Garin, Federica and Frasca, Paolo},
  title = {The Laplacian Spectrum of Large Graphs Sampled From Graphons},
  journal = {IEEE Transactions on Network Science and Engineering},
  volume = {8},
  number = {2},
  pages = {1711--1721},
  year = {2021},
  doi = {10.1109/TNSE.2021.3069675}
}

@inproceedings{GaoCaines2019,
  author = {Gao, Shuang and Caines, Peter E.},
  title = {Spectral Representations of Graphons in Very Large Network Systems Control},
  booktitle = {2019 IEEE 58th Conference on Decision and Control (CDC)},
  pages = {5068--5075},
  year = {2019},
  doi = {10.1109/CDC40024.2019.9030220}
}

@article{Tropp2015,
  author = {Tropp, Joel A.},
  title = {An Introduction to Matrix Concentration Inequalities},
  journal = {Foundations and Trends in Machine Learning},
  volume = {8},
  number = {1-2},
  pages = {1--230},
  year = {2015},
  doi = {10.1561/2200000048},
  eprint = {1501.01571},
  archivePrefix = {arXiv},
  primaryClass = {math.PR}
}
\end{document}